\documentclass[a4paper, twoside, 11pt]{article}

\usepackage{geometry}
\usepackage[english]{babel}
\usepackage[utf8]{inputenc}

\usepackage{authblk}

\usepackage{forest}
\usepackage{mathrsfs}
\usepackage{amsfonts, amssymb, amsmath, amsthm}
\usepackage{amsbsy}

\usepackage[backref=page]{hyperref}

\usepackage{tikz}
\usepackage{tabularx}

\usepackage{fancyhdr}
\usepackage{ifthen}
\newtheoremstyle{mystyle}{}{}{\rmfamily}%
{}{\normalfont\bfseries}{ }{ }{}

\hypersetup{colorlinks=true, linkcolor=blue, citecolor=blue, urlcolor=blue}

\theoremstyle{mystyle}
\newtheorem{theorem}{Theorem}[section]
\newtheorem*{theorem*}{Theorem}

\newtheorem{definition}[theorem]{Definition}

\newtheorem{corollary}[theorem]{Corollary}
\newtheorem{proposition}[theorem]{Proposition}
\newtheorem{remark}[theorem]{Remark}

\newtheorem{question}[theorem]{Question}
\newtheorem{example}[theorem]{\bf Example}

\newcommand{\keywords}[1]{\par\textbf{Keywords:} #1\par}

\newcommand{\subjclass}[1]{\par\noindent\textbf{MSC}: #1\par}

\usepackage{libertinus}

\title{
Dense Lineable Criterion for Linear Dynamics
}
\author[1]{Alexander Arbieto}
\author[2]{Manuel Saavedra}
\affil[1,2]{Instituto de Matemática, Universidade Federal do Rio de Janeiro, RJ, Brazil.}
\affil[1]{\texttt{arbieto@im.ufrj.br}}
\affil[2]{\texttt{saavmath@im.ufrj.br}}

\date{}

\fancypagestyle{customstyle}{
	\fancyhf{}
	\fancyhead[LE,RO]{\thepage}
	\fancyhead[CE,CO]{\ifthenelse{\isodd{\thepage}}{Alexander Arbieto and Manuel Saavedra }{Dense Lineable Criterion for Linear Dynamics}}
	
}

\begin{document}
\maketitle
\setlength{\headheight}{14.49998pt}
\begin{centering}
    \textit{``Puesta la mano en el pecho, resuelvo por fin obedecer tus leyes.''} — Ramiro Mendoza S., \\
  \hspace{4.9cm} grandfather of the first author, who dedicated his life to poetry.
\end{centering}

\begin{abstract}
	We study Li-Yorke chaos for sequences of continuous linear operators from an \(F\)-space to a normed space. We introduce the \emph{D-phenomenon} to establish a common dense lineable criterion that encompasses properties such as recurrence, universality, and Li-Yorke chaos. We show that in every infinite-dimensional separable complex Banach space, there exists a sequence of operators with a dense set of irregular vectors but without a dense irregular manifold, and we exhibit a recurrent operator whose set of recurrent vectors is not dense-lineable. This resolves in the negative a question posed by Grivaux et al.
\end{abstract}


\maketitle

\section{Introduction} 

Let \( X \) be a topological vector space (TVS). A subset \( A \subset X \) is said to be \textit{lineable} if \( A \cup \{0\} \) contains an infinite-dimensional vector subspace of \( X \). Moreover, \( A \) is called \textit{dense lineable} if \( A \cup \{0\} \) contains a dense vector subspace of \( X \). These concepts were introduced and extensively studied in \cite{gurariy2004lineability, aron2005lineability, seoane2006chaos}. For a comprehensive treatment of these and related notions, we refer to the works of L. Bernal et al. \cite{bernal2014lineability, bernal-linear}.

We introduce a motivating problem by recalling key aspects of quasi-rigid operators. Recently, the notion of recurrent operators was introduced by Costakis et al. \cite{Cos, Cos2}. Let \(X\) be an \(F\)-space. An operator \(T \in \mathcal{L}(X)\) is said to be \textit{recurrent} if the set of recurrent vectors for \(T\), denoted by \(\mathrm{Rec}(T)\), is dense in \(X\). Specifically, a vector \(x \in X\) is called recurrent for \(T\) if there exists a sequence \((\omega_n)_{n \in \mathbb{N}}\) with \(\omega_n \uparrow \infty\) such that \(\lim_{n} T^{\omega_n}x = x\). Furthermore, a continuous linear operator \(T \in \mathcal{L}(X)\) is said to be \textit{quasi-rigid} \cite{grivaux2025questions,manuel} if there exists a sequence \((\theta_n)_{n \in \mathbb{N}}\) with \(\theta_n \uparrow \infty\) such that the set \(\{x \in X : \lim_{n} T^{\theta_{n}}x = x\}\)
is dense in \(X\). Notably, if \(T\) is quasi-rigid, then \(\mathrm{Rec}(T)\) is dense lineable.

\begin{question}
Let \(X\) be an infinite dimensional separable \(F\)-space, and let \(\Lambda \subset \mathcal{L}(X)\) be a countable collection of quasi-rigid operators. Is it true 
\[
\bigcap_{T \in \Lambda} \mathrm{Rec}\left(T\right)
\]
is dense lineable?
\end{question}

As a consequence of our main Theorem \ref{mainthm}, we can affirmatively answer our motivating problem. Moreover, we can assert the existence of a dense vector subspace \(E \subset X\) such that, for each \(T \in \Lambda\), there exists a sequence \((k_{n,T})_{n} \uparrow \infty\) satisfying 
\begin{align*}
    E \subset \{x \in X : \lim_{n \rightarrow \infty} T^{k_{n,T}}x = x\}, \quad \forall\; T \in \Lambda.
\end{align*}

Recall that a continuous linear operator \( T: X \to X \) is said to be \textit{hypercyclic} if there exists a vector \( x \in X \) such that the set \( \{T^n x : n \in \mathbb{N} \} \) is dense in \( X \). The set of such vectors is denoted by \( \mathrm{HC}(T) \). 

On the other hand, a sequence of continuous linear operators \( (T_n: X \to Y)_{n \in \mathbb{N}} \) is called \textit{universal} if there exists a vector \( x \in X \) such that the set \( \{T_n x : n \in \mathbb{N} \} \) is dense in \( Y \). The set of such vectors is denoted by \( \mathrm{HC}((T_n)_{n \in \mathbb{N}}) \).

A significant distinction arises between studying the hypercyclicity of a single operator \( T \) and the universality of a sequence \( (T_{n})_{n} \). While the set of hypercyclic vectors \( \mathrm{HC}(T) \) is always dense in the hypercyclic case, the set \( \mathrm{HC}\left( (T_{n})_{n} \right) \) may not necessarily be dense in the universal case. This difference is highlighted by Godefroy and Shapiro \cite[pp. 234]{godefroy1991operators}.

Herrero \cite{herrero1991limits} and Bourdon \cite{bourdon1993invariant} demonstrated that if an operator \( T \) is hypercyclic on a complex Banach space, then the set of hypercyclic vectors \( \mathrm{HC}(T) \) is dense-lineable. For real Banach spaces, this result was extended by Bès \cite{bes1999invariant}, while for more general settings, such as topological vector spaces, Wengenroth \cite{wengenroth2003hypercyclic} further generalized the result.  In 1999, L. Bernal \cite{bernal1999densely} established two results concerning universal sequences and their connection to lineability and dense-lineability. Specifically, if \( X \) is a topological vector space, \( Y \) is a metrizable TVS, and \( (T_n : X \to Y)_{n \in \mathbb{N}} \) is hereditary hypercyclic, then \( \mathrm{HC}\left( (T_n)_{n} \right)\) is lineable. Moreover, if \( X, Y \) are separable metrizable spaces and \( (T_n : X \to Y)_{n \in \mathbb{N}} \) is densely hereditary hypercyclic (DHHC), then \( \mathrm{HC}\left( (T_n)_{n} \right) \) is dense-lineable. Recall that \((T_{n})\) is DHHC whenever every subsequence \(\mathrm{HC}((T_{n_{k}})_{k})\) is dense.

L. Bernal and Grosse-Erdmann established several equivalences related to the hypercyclicity criterion for a sequence of operators.

\begin{theorem}[Theorem 2.2 and Remark 2.3, \cite{bernal2003hypercyclicity}]\label{DHHC}
	Let \( X \) and \( Y \) be two separable \( F \)-spaces, and let \( (T_n)_{n} \subset \mathcal{L}(X, Y) \). Then, the following assertions are equivalent:
	\begin{enumerate}
		\item \( (T_n)_{n} \) satisfies the hypercyclicity criterion.
		\item \( (T_n)_{n} \) admits a densely hereditary hypercyclic subsequence.
		\item For each \( N \in \mathbb{N} \), the \( N \)-fold direct sum \( T_n \oplus \cdots \oplus T_n \) is densely hypercyclic.
	\end{enumerate}
\end{theorem}

F. León and V. Müller \cite{FernandoLeon2006} further developed the relationships between items \((1)\) and \((3)\) of the previous theorem, as well as additional implications, in the context of non-complete spaces. They also explored connections with dense-lineability. 

Recently, A. López and Q. Menet \cite{lopez2025two} provided a negative answer to an open problem proposed by S. Grivaux et al. \cite[Question 6.1]{grivaux2023questions}. The authors demonstrated that every infinite-dimensional separable complex Banach space admits a recurrent operator \(T\) for which the set of recurrent vectors, \(\mathrm{Rec}(T)\), is not dense-lineable.

An important notion in the study of linear dynamics is \textit{Li--Yorke chaos}. This concept was originally introduced by Li and Yorke \cite{li-yorke} in their seminal work on interval maps.

Let \((Z, d)\) be a metric space, and let \(f \colon Z \to Z\) be a continuous map. We say that a pair \((p, q) \in Z \times Z\) is a \textit{Li--Yorke pair} for \(f\) if:
\[
\liminf_{n \to \infty} d(f^n(p), f^n(q)) = 0 \quad \text{and} \quad \limsup_{n \to \infty} d(f^n(p), f^n(q)) > 0.
\]
A subset \(\Gamma \subset Z\) is called a \textit{scrambled set} for \(f\) if, for every pair of distinct points \(x, y \in \Gamma\), the pair \((x, y)\) is a Li--Yorke pair for \(f\). 

The map \(f\) is said to be \textit{Li--Yorke chaotic} if there exists an uncountable scrambled set for \(f\).

Bermúdez et al. \cite{bermudez2011li} investigated Li-Yorke chaos for operators on Banach spaces. They proved that Li-Yorke chaos is equivalent to the existence of irregular vectors and established sufficient conditions under which the set of irregular vectors is dense-lineable. These results were later extended by Bernardes et al. to the setting of Fréchet spaces \cite{bernardes2015li}. It is worth mentioning that Bernardes et al. \cite{bernardes2020mean} conducted a systematic study of Mean Li-Yorke Chaos, establishing criteria and sufficient conditions for the existence of a dense mean irregular manifold, followed by the work of \cite{jiang2025chaos}.

Irregular vectors have become a central topic in linear dynamics. In \cite{bernardes2013distributional}, Bernardes et al. studied distributional chaos in the context of Fréchet spaces. They introduced the \textit{Distributional Chaos Criterion}, which was shown to be equivalent to the existence of a distributionally irregular vector. Moreover, they provided conditions under which the set of uniform distributionally irregular vectors is dense-lineable. Recently, Conejero et al. \cite{conejero2016distributionally} explored distributional chaos for sequences of operators \((T_n)_{n} \in \mathcal{L}(X, Y)\), where \(X\) and \(Y\) are Fréchet spaces.

In this paper, motivated by the notion of universality, we introduce the concept of Li-Yorke chaos for sequences of continuous linear operators \((T_{n}: X \to Y)\), where \(X\) is an \(F\)-space and \(Y\) is a normed space, in Section 2. In Section 3, we introduce the notion of the \(D\)-phenomenon, study several properties, and review aspects of Furstenberg families. In Section 4, we establish a common dense-lineable criterion for sequences of operators. Additionally, we provide a negative answer to the open problem proposed by Grivaux et al. \cite[Problem 6.1]{grivaux2023questions}, and we also address \cite[Problem 1]{bernardes2015li} for sequences of operators, providing a negative answer.

\section{Li-Yorke Chaos for Sequences of Operators} \label{Section 2}

In this section, \(X\) denotes an \(F\)-space (that is, a completely metrizable topological vector space), \(Y\) a normed space, and \((T_{n} \colon X \to Y)_{n}\) a sequence of continuous linear operators.

A key notion in the study of Li--Yorke chaos, which can naturally be extended to sequences of operators, is that of irregular vectors.

\begin{definition}
    We say that a vector \(x \in X\) is \textit{irregular} for \((T_{n})_{n}\) if 
    \begin{align}
        \liminf_{n \to \infty} \Vert T_{n}(x) \Vert = 0 \quad \text{and} \quad 
        \limsup_{n \to \infty} \Vert T_{n}(x) \Vert = \infty.
    \end{align}
\end{definition}

In \cite{bermudez2011li, bernardes2015li}, it is established that, for a continuous linear operator \(T\) on a Banach or Fréchet space \(X\), the existence of semi-irregular vectors for \(T\) guarantees the existence of irregular vectors. Specifically, a vector \(x \in X\) is called \textit{semi-irregular} for \(T\) if it satisfies \(\liminf_{n \to \infty} \Vert{T^{n}x}\Vert = 0\) and \(\limsup_{n \to \infty} \Vert{T^{n}x}\Vert > 0\). 

\begin{example}
    Let \(X\) be a Banach space, and consider the sequence of operators \((T_{n})_{n \in \mathbb{N}} \in \mathcal{L}(X)\) defined by
    \[
        T_{n}(x) = 
        \begin{cases}
            0 & \text{if } n \text{ is odd}, \\
            x & \text{if } n \text{ is even}.
        \end{cases}
    \]
    For every nonzero vector \(x \in X \setminus \{0\}\), we observe that
    \begin{align*}
        \liminf_{n \to \infty} \Vert T_{n}(x) \Vert = 0 \quad \text{and} \quad 
        0 < \limsup_{n \to \infty} \Vert T_{n}(x) \Vert < \infty.
    \end{align*}
\end{example}

From the above example, it follows that every nonzero vector is semi-irregular for \((T_{n})_{n}\), yet no irregular vectors exist. This observation motivates the following definition of Li--Yorke chaos based on the notion of irregular vectors:

\begin{definition}
    A sequence of operators \((T_{n})_{n \in \mathbb{N}}\) is said to be \textit{Li--Yorke chaotic} if there exists an uncountable set \(\Gamma \subset X\) such that, for every pair \(p, q \in \Gamma\) with \(p \neq q\), the vector \(p - q\) is irregular for \((T_{n})_{n}\). Moreover, \((T_{n})_{n}\) is called \textit{densely Li--Yorke chaotic} if \(\Gamma\) is also dense in \(X\).
\end{definition}

The first observation is that if a sequence of operators \((T_{n})_{n}\) admits an irregular vector, then \((T_{n})_{n}\) is Li--Yorke chaotic. Indeed, if \(x\) is an irregular vector for \((T_{n})_{n}\), it suffices to consider the uncountable set \(\Gamma = \mathbb{C} \cdot x\). By construction, for every pair of distinct scalars \(\alpha, \beta \in \mathbb{C}\), the vector \(\alpha x - \beta x = (\alpha - \beta)x\) is irregular. This classical argument establishes the Li--Yorke chaoticity of \((T_{n})_{n}\).

\begin{example}
It is well known that every hypercyclic operator is Li--Yorke chaotic. Similarly, it can be shown, without much difficulty, that if a sequence of operators \((T_{n})_{n}\) is universal, then \((T_{n})_{n}\) is also Li--Yorke chaotic.
\end{example}

\begin{proposition}\label{G_de}
    Let \((T_{n})_{n \in \mathbb{N}} \subset \mathcal{L}(X, Y)\). The following sets:
    \begin{align*}
        A &= \{x \in X : \inf_{n \in \mathbb{N}} \Vert T_{n}x \Vert = 0\}, \\
        B &= \{x \in X : \sup_{n \in \mathbb{N}} \Vert T_{n}x \Vert = \infty\},
    \end{align*}
    are \(G_{\delta}\)-sets in \(X\). In particular, the set of irregular vectors for \((T_{n})_{n}\) is a \(G_{\delta}\)-set.
\end{proposition}

\begin{proof}
    To prove the result, we rewrite the sets \(A\) and \(B\) as countable intersections of open sets. Specifically, observe that:
    \begin{align*}
        A &= \bigcap_{k \in \mathbb{N}} \bigcup_{n \in \mathbb{N}} \{x \in X : \Vert T_{n}x \Vert < \tfrac{1}{k}\}, \\
        B &= \bigcap_{k \in \mathbb{N}} \bigcup_{n \in \mathbb{N}} \{x \in X : \Vert T_{n}x \Vert > k\}.
    \end{align*}
    Since each set \(\{x \in X : \Vert T_{n}x \Vert < \tfrac{1}{k}\}\) and \(\{x \in X : \Vert T_{n}x \Vert > k\}\) is open in \(X\), it follows that \(A\) and \(B\) are \(G_{\delta}\)-sets. Consequently, the set of irregular vectors for \((T_{n})_{n}\), which is the intersection \(A \cap B\), is also a \(G_{\delta}\)-set. This completes the proof.
\end{proof}

Bernardes et al., in \cite[Theorem 10]{bernardes2015li}, established six equivalent conditions characterizing when an operator \(T \colon X \to X\), defined on a separable Fréchet space \(X\), is densely Li--Yorke chaotic. Among these equivalences is the density of the set of irregular (or semi-irregular) vectors. This result motivates us to extend such characterizations to the context of sequences of operators, as formalized in the following theorem:

\begin{theorem} \label{equiv-den-LY}
    Let \(X\) be a separable \(F\)-space, and let \((T_{n} \colon X \to Y)_{n \in \mathbb{N}}\) be a sequence of continuous linear operators. Then the following assertions are equivalent:
    \begin{enumerate}
        \item \((T_{n})_{n}\) is densely Li--Yorke chaotic.
        \item \((T_{n})_{n}\) admits a dense set of irregular vectors.
        \item \((T_{n})_{n}\) admits a residual set of irregular vectors.
    \end{enumerate}
\end{theorem}

\begin{remark}
    In \cite[Theorem 10]{bernardes2015li}, a key implication is established: if \(T\) admits a residual set of irregular vectors, then \(T\) is densely Li--Yorke chaotic. Their proof relies on Zorn's Lemma and the construction of a \(D\)-subspace, where \(D := \mathbb{Q}\) or \(D := \mathbb{Q} + i\mathbb{Q}\), consisting (up to the zero vector) of irregular vectors for \(T\). Our approach to proving the implication \((3) \implies (1)\) in Theorem \ref{equiv-den-LY} differs from theirs.
\end{remark}

We will rely on Mycielski's Theorem \cite{mycielski1964independent, Tan}, which has been widely used in the study of nonlinear dynamical systems, particularly in the context of compact spaces \cite{Garcia, blanchard2002li, TanF, iwanik1991independence, huang2014stable, iwanik1989independent}. In the setting of linear dynamics, it has been applied by Zhen Jiang and Jian Li \cite{jiang2025chaos}. It is worth noting that these works involve Li-Yorke chaos. Beyond the context of chaoticity, M. Saavedra and M. Stadlbauer employed this theorem to establish an equivalence for quasi-rigid operators on separable \(F\)-spaces \cite[Theorem 3.3]{manuel}.

\begin{theorem}[Mycielski Theorem]
	Suppose that $X$ is a separable complete metric space without isolated points, and that for every $n\in \mathbb{N}$, the set $\mathcal{R}_{n}$ is residual in the product space $X^{n}$. Then
	there is a Mycielski set $\mathcal{K}$ in $X$ such that
	\begin{eqnarray*}
		(x_{1},x_{2}, \cdots, x_{n})\in \mathcal{R}_{n}
	\end{eqnarray*}
	for each $n\in \mathbb{N}$ and any pairwise different $n$ points $x_{1}, x_{2}, \ldots, x_{n}$ in $\mathcal{K}$.
\end{theorem}

A set $\mathcal{K}$ is referred to as a Mycielski
set if the intersection of $\mathcal{K}$ and any nonempty open set $U$ contains a Cantor set.

\begin{proof}[Proof of Theorem \ref{equiv-den-LY}]
Assume that (1) holds. Then, there exists a dense uncountable set \(\Gamma \subset X\) such that, for any \(p, q \in \Gamma\) with \(p \neq q\), the vector \(p - q\) is irregular. Fix \(q_{0} \in \Gamma\) and define the set \(D := \{p - q_{0} : p \in \Gamma \setminus \{q_{0}\}\}\). Since \(\Gamma\) is dense and uncountable, \(D\) is a dense set of irregular vectors. This proves (2).

To show that (2) implies (3), observe that the set of irregular vectors is a \(G_\delta\)-set by Proposition \ref{G_de}. Since this set is dense, it follows that it is residual, thereby establishing (3).

Finally, assume that (3) holds, that is, the set of irregular vectors is residual. Then, there exists a sequence of dense open sets \((U_{\ell})_{\ell \in \mathbb{N}}\) in \(X\) such that the set of irregular vectors is \(\bigcap_{\ell} U_{\ell}\). For each \(\ell \in \mathbb{N}\), define the dense open set \(W_{\ell} \subset X \times X\) by
\[
W_{\ell} := \{(p, q) \in X \times X : p - q \in U_{\ell}\}.
\]
The intersection \(\bigcap_{\ell} W_{\ell}\) is residual in \(X \times X\). By Mycielski's Theorem, there exists a Mycielski set \(\mathcal{K} \subset X\) such that
\[
\mathcal{K} \times \mathcal{K} \setminus \Delta \subset \bigcap_{\ell} W_{\ell},
\]
where \(\Delta\) denotes the diagonal in \(X \times X\). Let \(\Gamma := \mathcal{K}\). Clearly, \(\Gamma\) defined in this way is an uncountable dense set of \(X\) such that, for any distinct \(p, q \in \Gamma\), we have \(p - q \in \bigcap_{\ell} U_{\ell}\). In other words, \(p - q\) is an irregular vector for the sequence \((T_{n})_{n}\). This completes the proof.
\end{proof}

\begin{remark}
The implications \((1) \Rightarrow (2)\) and \((2) \Rightarrow (3)\) in the previous statement hold without requiring the space \(X\) to be separable.
\end{remark}

\begin{definition}\label{LYCC}
A sequence of operators \((T_n \colon X \to Y)_{n \in \mathbb{N}}\) is said to satisfy the \textit{Li-Yorke Chaotic Criterion} (LYCC) if there exist a subset \(X_{0} \subset X\) and a sequence of indices \((n_{k})_{k}\) such that the following properties hold:
\begin{enumerate}
    \item The sequence \((T_{n_{k}}x)_{k}\) converges to zero for every \(x \in X_{0}\).
    \item There exists a bounded sequence \((a_{n})_{n} \subset \overline{\mathrm{span}(X_{0})}\) such that \((T_{n}a_{n})_{n} \) is unbounded.
\end{enumerate}
\end{definition}

\begin{definition}
We say that a sequence of operators \((T_{n} \colon X \to Y)_{n \in \mathbb{N}}\) satisfies the \emph{dense Li--Yorke chaos criterion} if there exists a dense subset \(X_{0} \subseteq X\) such that \(X_{0}\) satisfies properties (1) and (2) of Definition \ref{LYCC}.
\end{definition}

\begin{remark}
In Theorem \ref{equi-LYCC}, we will prove that the Li--Yorke Chaotic Criterion (LYCC) is equivalent to Li--Yorke chaos for sequences of operators. These results are natural adaptations of Definition~7 and Theorem~8 in \cite{bermudez2011li} to the sequential setting.
\end{remark}

\begin{remark}
In \cite{bernardes2015li}, a (dense) Li--Yorke Chaos criterion is established for an operator \(T \colon X \to X\), where \(X\) is a Fréchet space. Condition (1) in \ref{LYCC} is given by the requirement that \(T^{n} x\) has a subsequence converging to zero for every \(x \in X_{0}\). They prove that the (dense) Li--Yorke Chaos criterion is equivalent to (dense) Li--Yorke chaos, respectively.
\end{remark}

\begin{remark}
In Theorem \ref{DLYCC-produc}, we characterize the dense Li-Yorke Chaos criterion. However, in general, this criterion is not equivalent to densely Li--Yorke chaotic, as we will show in Theorem \ref{respuesta}. This represents a significant departure from the results obtained in \cite{bernardes2015li}.
\end{remark}

\begin{theorem}\label{equi-LYCC}
A sequence of operators \((T_n \colon X \to Y)_{n \in \mathbb{N}}\) is Li--Yorke chaotic if and only if it satisfies the Li--Yorke Chaotic Criterion (LYCC).
\end{theorem}

\begin{proof}
If \((T_n)_{n \in \mathbb{N}}\) is Li-Yorke chaotic, then there exists an irregular vector \(q \in X\) for \((T_n)_n\). In this case, it suffices to take \(X_0 = \{q\}\).

Conversely, if some vector in \(X_0\) is irregular for \((T_n)_n\), the conclusion follows immediately. Otherwise, suppose that no vector in \(X_0\) is irregular. Then, for every \(u \in \mathrm{span}(X_0)\), we have:
\begin{align*}
    \lim_{k \to \infty} \Vert T_{n_k}u \Vert = 0 \quad \text{and} \quad \sup_{n \in \mathbb{N}} \Vert T_n u \Vert < \infty.
\end{align*}
By applying an inductive construction, we obtain a subsequence \((\theta_k)_k\) of \((n_k)_k\), a sequence \((m_k)_k \subset \mathbb{N}\) and a sequence of vectors \(\{u_k\}_k \subset \mathrm{span}(X_0)\) satisfying the following conditions for all \(k \in \mathbb{N}\):
\begin{itemize}
    \item \( u_{k} \in B(0, 2^{-k})\),
    \item \(\Vert T_{\theta_k}u_i \Vert < 2^{-i}k^{-1}\) for all \(1 \leq i \leq k\),
    \item For \(k > 1\),
    \[
        \Vert T_{m_k}u_k \Vert > k + \sum_{i=1}^{k-1} \sup_{n \in \mathbb{N}} \Vert T_n u_i \Vert.
    \]
\end{itemize}
Inductively, we can find a strictly increasing sequence of positive integers \((i_{\ell})_{\ell}\) such that:
\begin{itemize}
    \item \(i_{\ell+1} > \max\{m_{i_{\ell}}, \theta_{i_{\ell}}\}\), and
    \item For each \(j \in \{1, \ldots, i_{\ell}\}\),
    \begin{align*}
        \sup_{x \in B(0, 2^{-i_{\ell+1}})} \Vert{T_{j}x}\Vert < 2^{-\ell}.
    \end{align*}
\end{itemize}

Certainly, \(p := \sum_{\ell \in \mathbb{N}} u_{i_{\ell}} \in X\). We claim that the vector \(p\) is irregular for \((T_{n})_{n}\). 

We observe that 
\begin{align*}
    \Vert{T_{m_{i_{t}}}p}\Vert & \geq \Vert{T_{m_{i_{t}}}u_{i_{t}}}\Vert - \sum_{\ell < t} \Vert{T_{m_{i_{t}}}u_{i_{\ell}}}\Vert - \sum_{\ell > t} \Vert{T_{m_{i_{t}}}u_{i_{\ell}}}\Vert \\
    & > i_{t} + \sum_{s < i_{t}} \sup_{n \in \mathbb{N}} \Vert{T_{n}u_{s}}\Vert - \sum_{\ell < t} \sup_{n \in \mathbb{N}} \Vert{T_{n}u_{i_{\ell}}}\Vert - \sum_{\ell > t} 2^{-\ell} \\
    & > i_{t} - 1 \xrightarrow[t \rightarrow \infty]{} \infty
\end{align*}
On the other hand, 
\begin{align*}
    \Vert{T_{\theta_{i_{t}}}p}\Vert & \leq \sum_{\ell \leq t} \Vert{T_{\theta_{i_{t}}}u_{i_{\ell}}}\Vert + \sum_{\ell > t} \Vert{T_{\theta_{i_{t}}}u_{i_{\ell}}}\Vert  \\
    {} & \leq \frac{1}{i_{t}} \sum_{\ell \leq t} 2^{-i_{\ell}} + \sum_{\ell > t} 2^{-\ell}\\
    {} & \leq  \frac{1}{i_{t}}+ 2^{-t} \xrightarrow[t \rightarrow \infty]{}0.
\end{align*}
This completes the proof.
\end{proof}

\begin{proposition}\label{DLY to dense}
If \((T_{n}:X \rightarrow Y)_{n}\) satisfies the Dense Li-Yorke Chaotic Criterion, then it is densely Li-Yorke chaotic.
\end{proposition}
\begin{proof}
It suffices to show that the set of irregular vectors is dense, as guaranteed by Theorem \ref{equiv-den-LY}. Assume, for contradiction, that the set of irregular vectors for \((T_{n})_{n}\), denoted by \(\mathcal{I}\), is not dense in \(X\). Then, there exists a non-empty open set \(U \subset X\) such that \(U \cap \mathcal{I} = \emptyset\).

By hypothesis, there exists a dense set \(X_{0} \subset X\) and a sequence of indices \((n_{k})_{k}\) satisfying conditions (1) and (2) in Definition \ref{LYCC}. By the density of \(X_{0}\), there exists a vector \(q \in U \cap X_{0}\). Since \(q\) is not irregular, it follows that
\[
T_{n_{k}}q \xrightarrow[k \to \infty]{} 0 \quad \text{and} \quad M := \sup_{n} \Vert{T_{n}q}\Vert < \infty.
\]

In the proof of Theorem \ref{equi-LYCC}, we construct an irregular vector \(p \in X\), a subsequence \((\theta_{i_{t}})\) of \((n_{k})_{k}\), and a sequence \((m_{i_{t}})\). For any non-zero complex number \(\lambda\), we claim that the vector \(q + \lambda p\) is irregular. This follows from the following properties:
\[
T_{\theta_{i_{t}}}(q + \lambda p) \xrightarrow[t \to \infty]{} 0,
\]
and
\[
\Vert{T_{m_{i_{t}}}(q + \lambda p)}\Vert \geq \vert{\lambda}\vert(i_{t} - 1) - M \xrightarrow[t \to \infty]{} \infty.
\]

Clearly, there exists some \(\lambda_{0} \neq 0\) such that \(q + \lambda_{0}p \in U\). This contradicts the assumption that \(U \cap \mathcal{I} = \emptyset\).
\end{proof}

\begin{theorem}\label{DLYCC-produc}
Let \(X\) be a separable \(F\)-space, \(Y\) a normed space, and \((T_{n}: X \rightarrow Y)_{n} \subset \mathcal{L}(X, Y)\). Then, the following statements are equivalent:
\begin{enumerate}
    \item \((T_{n})_{n}\) satisfies the Dense Li-Yorke Chaotic Criterion.
    \item For each \(m \in \mathbb{N}\), the \(m\)-fold direct sum 
    \[
    T_{n} \oplus \cdots \oplus T_{n}: X^{m} \longrightarrow Y^{m}
    \]
    is densely Li-Yorke chaotic.
\end{enumerate}
\end{theorem}

\begin{proof}
If \((T_{n})_{n}\) satisfies the Dense Li-Yorke Chaotic Criterion, it is clear that for each \(m \in \mathbb{N}\), the \(m\)-fold direct sum \(T_{n} \oplus \cdots \oplus T_{n}: X^{m} \to Y^{m}\) also satisfies the criterion. According to Proposition \ref{DLY to dense}, this implies (2).

To show that (2) implies (1), it suffices to prove the existence of a sequence \((n_{k})_{k}\) and a dense subset \(X_{0} \subset X\) such that \(T_{n_{k}}x \to 0\) for every \(x \in X_{0}\).

For each \(m \in \mathbb{N}\), define \(\mathcal{R}_{m} \subset X^{m}\) as:
\[
\mathcal{R}_{m} := \{(x_{1}, \ldots, x_{m}) \in X^{m} : \exists \theta_{\ell} \uparrow \infty \; \text{such that} \; \lim_{\ell} T_{\theta_{\ell}}x_{i} = 0, \forall i \in \{1, \ldots, m\}\}.
\]
Clearly, \(\mathcal{R}_{m}\) is a \(G_{\delta}\)-set and is dense since it contains the set of irregular vectors.

By Mycielski's Theorem, there exists a Mycielski set \(\mathcal{K} \subset X\) such that \(\mathcal{K}^{m} \subset \mathcal{R}_{m}\) for every \(m \in \mathbb{N}\). Since \(X\) is separable, there exists a countable dense set \(X_{0} := \{q_{i}\}_{i \in \mathbb{N}} \subset \mathcal{K}\). We can then choose a strictly increasing sequence of positive integers \((n_{k})_{k}\) such that \(\lim_{k} T_{n_{k}}x = 0\) for every \(x \in X_{0}\).
\end{proof}

We say that \(T:X \to X\) admits a dense irregular manifold if there exists a dense subspace where every non-zero vector is irregular for \(T\); in other words, if the set of irregular vectors is dense-lineable. In \cite{bermudez2011li, bernardes2015li, jiang2025chaos}, sufficient conditions for the existence of a dense irregular manifold are established.

\begin{question}[\cite{bernardes2015li}]
Does dense Li–Yorke chaos imply the existence of a dense irregular manifold for operators on Fréchet (or Banach) spaces?
\end{question}

It is natural to ask whether a sequence of operators \((T_{n}: X \to X)_{n}\) that is densely Li-Yorke chaotic necessarily admits a dense irregular manifold. We answer this question in the negative for every infinite-dimensional separable complex Banach space \(X\) in Theorem \ref{respuesta}.

\section{D-phenomenon and Furstenberg family} \label{section 3}

In linear dynamics, there are several behaviors of interest, such as recurrence, Li-Yorke chaos, hypercyclicity, among others. Clearly, seeking a dense-lineable criterion for each of these aspects individually is not particularly motivating. In this section, we introduce the notion of the \emph{\(D\)-phenomenon}, which allows us to study the aforementioned behaviors from a unified perspective. Moreover, it enables us to establish a common dense-lineable criterion for a countable family of \(D\)-phenomena, as we will see in the next section.

\begin{definition}
	 Let \( X \) and \( Y \) be topological vector spaces. We say that \( \Psi \) is a \emph{D-phenomenon} from \( X \) to \( Y \) if, for each \( x \in X \), there exists a non-empty set \( \Psi(x) \) consisting of open subsets of \( Y \) such that the following holds: for any \( p = \sum_{j=1}^{m} \alpha_j x_j \in X \) with \(\alpha_{i}\neq 0\) and for any \( V \in \Psi(p) \), there exist non-empty open sets \( W_i \in \Psi(x_i) \) such that 
    \[
    \sum_{j=1}^{m} \alpha_j W_i \subset V.
    \]
For the \( m \)-fold product \( X^m \), we define
	\[
	\Psi(x_1, \ldots, x_m) := \left\{ \prod_{i=1}^{m} A_i \subset X^m : A_i \in \Psi(x_i) \right\}.
	\]
\end{definition}

\begin{remark}
    The choice of the term "D-phenomenon" refers to a dynamic phenomenon. Throughout this section, the connection to dynamic aspects, such as recurrence, hypercyclicity, and others, will become more apparent. 
\end{remark}

\begin{example}\label{ex-Lin}
    The following are key examples of D-phenomena in linear dynamics.

    \begin{itemize}
        \item 
        Let \(X\) be a topological vector space (TVS). For each \(x \in X\), define 
        \begin{equation}\label{Li-rec}
            \Psi^{\mathrm{Rec}}(x) := \left\{ \text{all open neighborhoods of } x \text{ in } X \right\}.
        \end{equation}

        \item  
        Let \(X\) and \(Y\) be topological vector spaces. For each \(x \in X\), define 
        \begin{equation}\label{Li-hyper}
            \Psi^{\mathrm{HC}}(x) := \left\{ \text{all non-empty open subsets of } Y \right\}.
        \end{equation}

        \item   
        Let \(X, Y\) be topological vector spaces, and let \(h \colon X \to Y\) be a continuous linear operator. For each \(x \in X\), define 
        \begin{equation}\label{Li-h}
            \Psi^{h}(x) := \left\{ \text{all open neighborhoods of } h(x) \text{ in } Y \right\}.
        \end{equation}
    \end{itemize}

    In all cases, one can verify that \(\Psi\) defines a \emph{\(D\)-phenomenon}.
\end{example}

\begin{proposition}\label{union}
    Let \(X\) and \(Y\) be topological vector spaces, and let \(\{\Psi_i\}_{i \in J}\) be a family of \emph{\(D\)-phenomena} from \(X\) to \(Y\). Define the union \(D\)-phenomenon \(\Psi_J\) as 
    \[
        \Psi_J(x) := \bigcup_{i \in J} \Psi_i(x) = \left\{ V \subseteq Y : V \in \Psi_i(x) \ \text{for some} \ i \in J \right\}.
    \]
    Then, \(\Psi_J\) is a \emph{\(D\)-phenomenon} from \(X\) to \(Y\).
\end{proposition}

\begin{proof}
    Let \(p = \sum_{j=1}^m \alpha_j x_j\), where \(\alpha_j\) are non-zero complex numbers and \(x_j \in X\). Fix an open set \(V \in \Psi_J(p)\). By definition of \(\Psi_J\), there exists \(i \in J\) such that \(V \in \Psi_i(p)\). 

    Since \(\Psi_i\) is a \(D\)-phenomenon, there exist non-empty open sets \(W_j \in \Psi_i(x_j)\) for each \(j = 1, \ldots, m\), satisfying:
    \[
        \sum_{j=1}^m \alpha_j W_j \subseteq V.
    \]
    The key observation is that \(W_j \in \Psi_J(x_j)\) for all \(j\), as \(\Psi_J\) contains all sets from each \(\Psi_i\). Thus, \(\Psi_J\) inherits the \(D\)-phenomenon property from the family \(\{\Psi_i\}_{i \in J}\).
\end{proof}

We observe that the \(D\)-phenomena presented in (\ref{Li-rec}) and (\ref{Li-h}) of Example~\ref{ex-Lin} are associated with single operators. In contrast, the phenomenon in (\ref{Li-hyper}) corresponds to a union of \(D\)-phenomena (as in Proposition~\ref{union}), where each constituent \(D\)-phenomenon is linked to an operator in \(\mathcal{L}(X,Y)\). This raises a natural question: Is every \(D\)-phenomenon expressible as a union of operator-associated \(D\)-phenomena? The following result addresses this inquiry.

\begin{proposition}\label{prop-inverse-phenomenon}
Let \(h: X \rightarrow Y\) be a surjective operator between two TVSs \(X\) and \(Y\). For each \(x \in X\), define
\begin{equation}\label{Li-inv}
    \Psi(x) := \left\{ h^{-1}(U) \subseteq X : U \text{ is an open neighborhood of } h(x) \text{ in } Y \right\}.
\end{equation}
Then, \(\Psi\) is a \emph{\(D\)-phenomenon} from \(X\) to \(X\).
\end{proposition}

\begin{proof}
Let \(p = \sum_{i=1}^m \alpha_i x_i\) with \(\alpha_i \neq 0\) and \(x_i \in X\). Fix \(V = h^{-1}(U) \in \Psi(p)\), where \(U \subseteq Y\) is open and \(h(p) \in U\). Since \(h\) is linear and continuous, we have:
\[
h(p) = \sum_{i=1}^m \alpha_i h(x_i) \in U.
\]
By the continuity of vector addition and scalar multiplication in \(Y\), there exist open neighborhoods \(W_i \subseteq Y\) of \(h(x_i)\) such that \(\sum_{i=1}^m \alpha_i W_i \subseteq U\).

For each \(i\), define \(Z_i := h^{-1}(W_i) \in \Psi(x_i)\). Then:
	\[
	\sum_{i=1}^m \alpha_i Z_i \subseteq h^{-1}(U) = V.
	\]
	Thus, \(\Psi\) satisfies the \(D\)-phenomenon property.
\end{proof}

\begin{corollary}\label{unbounded}
    Let \(X\) be an infinite-dimensional Banach space and let \(h \in \mathcal{L}(X)\) be a surjective bounded operator that is not injective. Then, the D-phenomenon \(\Psi\) in (\ref{Li-inv}) for \(h\) is not a union of D-phenomena associated with continuous operators.
\end{corollary}

\begin{proof}
We can notice that every D-phenomenon generated by the union of a family of D-phenomena associated with continuous operators admits bounded open sets. In contrast,  every open set in \(\Psi(0)\) is unbounded
\end{proof}

The following result provides a characterization of how bounded sets propagate through \(D\)-phenomena, revealing that the local existence of bounded open sets implies their ubiquity:

\begin{proposition}\label{prop-boundedness}
Let \(X\) and \(Y\) be topological vector spaces, and \(\Psi\) a \(D\)-phenomenon from \(X\) to \(Y\). The following are equivalent:
\begin{enumerate}
    \item[(i)] \(\Psi(x)\) contains a bounded open set for every \(x \in X\);
    \item[(ii)] \(\Psi(q)\) contains a bounded open set for some \(q \in X\).
\end{enumerate}
\end{proposition}

\begin{proof}
The implication (i) \(\Rightarrow\) (ii) is immediate. For the converse, suppose there exists \(q \in X\) with a bounded open set \(U \in \Psi(q)\). Let \(x \in X\) be arbitrary. Decompose \(q = x + (q - x)\) and observe that:

By the \(D\)-phenomenon property, there exist open sets \(A \in \Psi(x)\) and \(B \in \Psi(q - x)\) satisfying \(A + B \subseteq U\). Clearly, this implies that \(A\) is a bounded open set. Thus \(\Psi(x)\) contains the bounded open set \(A\).
\end{proof}

\begin{definition}\label{def:bounded-D}
A \(D\)-phenomenon \(\Psi\) from \(X\) to \(Y\) is called \emph{bounded} if \(\Psi(0)\) contains a bounded open set.
\end{definition}

\begin{example}\label{ex:metric-bounded}
Let \(X\) be a Banach space. For each \(x \in X\), define
\[
\Psi(x) := \left\{ B_{X}\left(x, \tfrac{1}{n}\right) : n \in \mathbb{N} \right\},
\]  
where \(B_X(x, \tfrac{1}{n})\) denotes the open ball of radius \(1/n\) centered at \(x\). Then \(\Psi\) is a bounded \(D\)-phenomenon, refining the structure introduced in (\ref{Li-rec}).
\end{example}

This example reveals natural structural relationships between \(D\)-phenomena. To formalize this, we define a partial order on the class of \(D\)-phenomena from \(X\) to \(Y\):
\begin{itemize}
    \item For \(D\)-phenomena \(\Psi_1\) and \(\Psi_2\), we write \(\Psi_1 \prec \Psi_2\) if for every \(x \in X\) and every \(V \in \Psi_1(x)\), there exists \(W \in \Psi_2(x)\) with \(W \subseteq V\).
    \item We write \(\Psi_1 \sim \Psi_2\) if both \(\Psi_1 \prec \Psi_2\) and \(\Psi_2 \prec \Psi_1\).
\end{itemize}

A \(D\)-phenomenon \(\Psi\) is called \emph{minimal} if, whenever \(\Psi_1 \prec \Psi\) for some \(D\)-phenomenon \(\Psi_1\), it follows that \(\Psi_1 \sim \Psi\).

\begin{proposition}\label{prop:minimal-bounded}
Let \(X\) and \(Y\) be topological vector spaces, where \(Y\) has a bounded neighborhood basis. For every \(h \in \mathcal{L}(X,Y)\), the \(D\)-phenomenon \(\Psi^h\) (defined in (\ref{Li-h})) is minimal among all bounded \(D\)-phenomena under the partial order \(\prec\).
\end{proposition}

\begin{proof}
To establish the minimality of \(\Psi^h\), let \(\Psi\) be any bounded \(D\)-phenomenon satisfying \(\Psi \prec \Psi^h\). We will prove that \(\Psi^h \prec \Psi\).  

Since \(\Psi \prec \Psi^h\), we observe that for every \(x \in X\), each open set in \(\Psi(x)\) contains the vector \(h(x)\).  

Fix \(q \in X\) and consider \(V \in \Psi^h(q)\). Define \(M := V - h(q)\), which is an open neighborhood of \(0\) in \(Y\).  

Since \(\Psi\) is a bounded \(D\)-phenomenon, there exists a bounded open set \(K \in \Psi(0)\), and for some \(m \in \mathbb{N}\) we have that \(K \subseteq m \cdot M\). On the other hand, since \(m \cdot 0 = 0\), there exists \(U_1 \in \Psi(0)\) with \(m \cdot U_1 \subseteq K\), implying that \(U_1 \subseteq M\).  

Writing \(0 = q + (-q)\), the \(D\)-phenomenon guarantees the existence of two open sets \(W \in \Psi(q)\) and \(U_2 \in \Psi(-q)\) such that \(W + U_2 \subseteq U_1\). Observing that \(U_2 + h(q)\) is an open neighborhood of \(0\), we obtain:
\[
W + (U_2 + h(q)) \subseteq M + h(q) = V.
\]
Thus, \(W \subseteq V\), proving that \(\Psi^h \prec \Psi\). Since \(\Psi\) was arbitrary, the minimality of \(\Psi^h\) follows.
\end{proof}

\begin{figure}[h]
\centering
\begin{forest}
for tree={edge=->, grow'=north, math content, l sep=2.5cm, s sep=1.5cm, scale=0.8}
[
  \Psi^{\mathrm{HC}} \quad (\text{Maximal}), name=max
  [
    \begin{array}{c}
      \text{Bounded } D\text{-phenomena} \\
      \text{(e.g., } \Psi \text{ de Prop.\ref{ex:metric-bounded})}
    \end{array}, 
    [
      \Psi^{h}\,(\text{Minimal}), 
    ]
    [
      \Psi^{g}\,(\text{Minimal}), 
    ]
    [
    \text{Prop.}\, \ref{D-phe. LY} , 
  ]
    [
    \begin{array}{c}
      \bullet \\
      \bullet \\
      \bullet
    \end{array},
    ]
  ]
  [
    \text{Cor.}\, \ref{unbounded} , 
  ]
  [
    \begin{array}{c}
      \bullet \\
      \bullet \\
      \bullet
    \end{array},
  ]
]
\end{forest}
\caption{Diagram of the partial order of D-phenomena}
\label{fig:D-phenomena-hierarchy}
\end{figure}
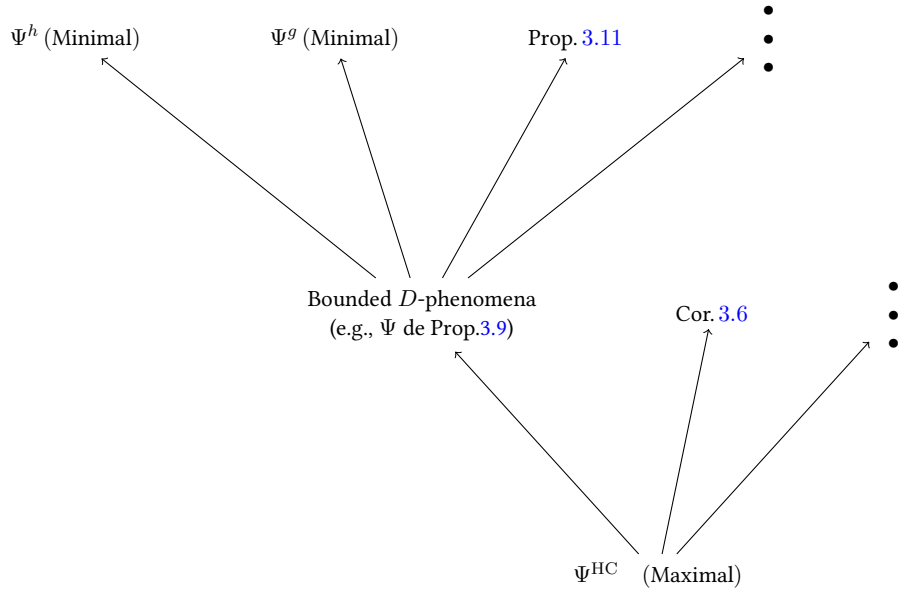

Up to this point, all studied \(D\)-phenomena have required that every open set \(V \subset Y\) in \(\Psi(0)\) contains the zero vector. The following proposition establishes a fundamentally different behavior, which is intrinsically linked to Li--Yorke chaos.

\begin{proposition}\label{D-phe. LY}
    Let \(X\) be a topological vector space and \(Y\) be a normed space. For each \(x \in X\), define
   \begin{align}
       \Psi^{\mathrm{LY}}(x) := \{B(0, r), Y \setminus B[0, s] : r,s>0\}
   \end{align}
    Then \(\Psi^{\mathrm{LY}}\) is a D-phenomenon from \(X\) to \(Y\).
\end{proposition}

\begin{proof}
Let \(p = \sum_{i=1}^{t} \alpha_{i} x_{i}\) be any fixed vector in \(X\) with \(\alpha_{i}\neq 0\). We distinguish two cases based on the open set in \(\Psi(p)\):

\textbf{Case 1:} Consider an open set \(B(0, r) \in \Psi(p)\) for some \(r>0\). Clearly, there exists \(n \in \mathbb{N}\) such that \(\sum_{i} |\alpha_{i}| < nr\). Now, consider the open sets \(W_{i} := B(0, \frac{1}{n}) \in \Psi(x_{i})\). Then, 
\[
\sum_{i} \alpha_{i} W_{i} \subset B(0, r).
\]

\textbf{Case 2:} Consider an open set \(Y \setminus B[0, s] \in \Psi(p)\) for some \(s \in \mathbb{N}\). We choose the open set \(W_{1} := Y \setminus B[0, \ell]\), where \(\ell \in \mathbb{N}\) satisfies \(|\alpha_{1}| \ell > s + 1\). For \(i \in \{2, \ldots, t\}\), choose the open sets \(W_{i} := B(0, \frac{1}{n})\) in \(\Psi(x_{i})\) such that \(n \in \mathbb{N}\) and \(\sum_{i=2}^{t} |\alpha_{i}| < n\). Then, it follows that
\[
\sum_{i} \alpha_{i} W_{i} \subset Y \setminus B[0, s].
\]

Therefore, \(\Psi\) satisfies the requirements of a D-phenomenon.
\end{proof}

To conclude this section, we will establish connections between \(D\)-phenomena and the dynamical properties of an operator, such as recurrence, Li-Yorke chaos, among others. In order to do so, we recall some notions.

The notion of a Furstenberg family was introduced by Akin \cite{akin2013recurrence}, inspired by earlier works such as \cite{Furs, gotts}. 

\begin{definition}
	A non-empty family \(\mathcal{F}\) of subsets of \(\mathbb{N}_0\) is called a \textit{Furstenberg family} if it is hereditary upward, meaning that 
	\begin{align*}
		A \in \mathcal{F}, \, A \subset B \implies B \in \mathcal{F}.
	\end{align*}
	Moreover, a Furstenberg family is said to be \textit{proper} if it does not contain the empty set.
\end{definition}

Let \(\mathcal{F}\) be a Furstenberg family, we say \((T_{n}:X\rightarrow Y)_{n\in \mathbb{N}}\) is \(\mathcal{F}\)-universal if there exists a point \(x \in X\) such that, for any non-empty open set \(U\) in \(Y\),
\[
\{n \in \mathbb{N} \colon T_n(x) \in U\} \in \mathcal{F}.
\]
The point \(x\) is then called \(\mathcal{F}\)-universal for \((T_{n})_{n}\), and the set of such vectors is denoted by \(\mathcal{F}\mathrm{HC}((T_{n})_{n})\).

We can distinguish several central notions in the study of hypercyclic operators. More precisely, these include frequently hypercyclic, upper frequently hypercyclic, and reiteratively hypercyclic operators. These concepts have been extensively studied in the literature; see, for instance, \cite{bayart2004hypercyclicite, bayart2006frequently, bayart2015difference, bes2016recurrence, bonillaupper, Boni, charpentier2022common, ernst, grivaux2018frequently, grivaux2014invariant, menet2017linear}. It is worth noting that the notion of \(\mathcal{F}\)-hypercyclicity was introduced by Shkarin \cite[Sec. 5]{shkarin2009spectrum}.

Now, let \( X, Y \) be two topological vector spaces, \( \Psi \) a \(D\)-phenomenon from \( X \) to \( Y \), and \( \mathcal{F} \) a proper Furstenberg family. We say that \( x \in \mathcal{F}\Psi((T_{n})_{n}) \) if, for every \( V \in \Psi(x) \), 
\begin{align*}
    \{n \in \mathbb{N} : T_n(x) \in V\} \in \mathcal{F}.
\end{align*}  

We say that \( x \in \Psi((T_{n})_{n}) \) if, for every \( V \in \Psi(x) \), 
\begin{align*}
    \{n \in \mathbb{N} : T_n(x) \in V\} \neq \emptyset
\end{align*}

\begin{enumerate}
    \item \textbf{Recurrence.} Let \( X \) be an \( F \)-space and \( T \in \mathcal{L}(X) \). Considering the \( D \)-phenomenon \(\Psi^{\mathrm{Rec}}\) from \( X \) to \( X \), we have  
    \begin{align*}
        \mathcal{F}\mathrm{Rec}(T) = \mathcal{F}\Psi^{\mathrm{Rec}}(T).
    \end{align*}
    
    \item \textbf{Universality.} Let \( X \) be an \( F \)-space and \( Y \) a metrizable topological vector space, and let \((T_{n})_{n} \subset \mathcal{L}(X, Y)\). Considering the \( D \)-phenomenon \(\Psi^{\mathrm{HC}}\) from \( X \) to \( Y \), we obtain  
    \begin{align*}
        \mathcal{F}\mathrm{HC}((T_{n})_{n}) = \mathcal{F}\Psi^{\mathrm{HC}}((T_{n})_{n}).
    \end{align*}
    
    \item \textbf{Li-Yorke Chaos.} Let \( X \) be an \( F \)-space, \( Y \) a normed space, and \((T_{n})_{n} \subset \mathcal{L}(X, Y)\) a sequence of operators. Considering the \( D \)-phenomenon \(\Psi^{\mathrm{LY}}\), we have  
    \begin{align*}
        x \, \text{is irregular for} \, (T_{n})_{n} \, \text{if and only if} \, x \in \Psi^{\mathrm{LY}}((T_{n})_{n}).
    \end{align*}
\end{enumerate}

For the \( k \)-fold product, \( q = (x_1, \ldots, x_k) \in \mathcal{F}\Psi((T_{n})_{n}:k) \) if, for every \(V\in \Psi(q) \),
\begin{align}\label{k-fold}
    \{n \in \mathbb{N} : (T_n(x_1), \ldots, T_n(x_k)) \in V\} \in \mathcal{F}.
\end{align}

\begin{remark}
It is important to pay attention to (\ref{k-fold}). In the cases of recurrence and universality, we have
\begin{align*}
    \mathcal{F}\Psi^{\mathrm{Rec}}(T; k) &= \mathcal{F}\mathrm{Rec}\left(\bigoplus_{i=1}^{k} T\right), \\
    \mathcal{F}\Psi^{\mathrm{HC}}((T_{n})_{n}; k) &= \mathcal{F}\mathrm{HC}\left(\bigoplus_{i=1}^{k} T_{n}\right).
\end{align*}
However, this equality does not necessarily hold when considering \(\Psi^{\mathrm{LY}}\).
\end{remark}

\section{Commom Dense-Lineable Criterion} \label{section 4}

As reviewed in the introduction, the set \(\mathrm{HC}(T)\) is known to be dense-lineable when \(T\) is a hypercyclic operator. However, in the context of recurrent operators, certain aspects differ significantly.

Let \(X\) be an infinite-dimensional separable Banach space, and consider a recurrent operator \(T: X \to X\), known as an Augé-Tapia operator of type \(N\) \cite{Tapia, manuel}. To delve into the details, we first recall some relevant aspects of this class of operators.

By a theorem of Ovsepian and Aleksander \cite{Ovsepian}, there exists a sequence \( (e_{k}, e_{k}^{*})_{k \in \mathbb{N}} \subset X \times X^{*} \) such that \( \mathrm{span}\{e_{n} : n \in \mathbb{N}\} \) is dense in \( X \), the functionals satisfy \( e_{n}^{*}(e_{m}) = \delta_{n,m} \), each vector has norm \( \|e_{n}\| = 1 \), and the functionals are uniformly bounded, i.e., \( \sup_{n \in \mathbb{N}} \|e_{n}^{*}\| < \infty \).

Consider \( V := \mathrm{span}\{e_{1}, \ldots, e_{N}\} \), and denote by \( \mathbb{P} \) the projection from \( X \) onto \( V \), defined as
\[
\mathbb{P}(x) =\sum_{i=1}^{N} \langle e_{i}^{*}, x \rangle e_{i} 
\]

On the other hand, there exists \( F \subset V \) such that both \( F \) and \( V \setminus F \) are dense sets in \( V \) \cite{Tapia}, where
\[
\mathbb{P}^{-1}(F) = \mathrm{Rec}(T) \quad \text{and} \quad A_{T} := \mathbb{P}^{-1}(V \setminus F) = \{x \in X : \lim_{n} \|T^{n}x\| = \infty\}.
\]
In other words, \(T\) is recurrent, and \(X = \mathrm{Rec}(T) \cup A_{T}\), with both sets being dense in \(X\).

\begin{theorem}\label{rec-notdense}
Let \(X\) be an infinite-dimensional separable Banach space. Then there exists \(T \in \mathcal{L}(X)\) such that \(T\) is recurrent and \(\mathrm{Rec}(T)\) is not dense-lineable.
\end{theorem}

\begin{proof}
Let \(T\) be an Áuge-Tapia operator of type \(N\), which is known to be recurrent. We claim that \( \mathrm{Rec}(T) \) is not dense-lineable, following the ideas in the proof of \cite[Thm. 3.11]{manuel}. 

Assume, for contradiction, that \( \mathrm{Rec}(T) \) is dense-lineable. Then there exists a dense vector subspace \( E \subset X \) such that \( E \subset \mathrm{Rec}(T) \), leading to the chain of inclusions:
\[
V = \mathbb{P}(E) \subset \mathbb{P}(\mathrm{Rec}(T)) = F \varsubsetneq V,
\]
which is a contradiction.
\end{proof}

The previous result provides a negative answer to the open problem proposed by Grivaux et al. \cite[Question 6.1]{grivaux2023questions}. This approach is independent of the result obtained by A. López and Q. Menet \cite[Theorem 2.1]{lopez2025two}.

\begin{theorem}\label{respuesta}
Let \(X\) be an infinite-dimensional separable complex Banach space. There exists a sequence \((T_{n})_{n} \subset \mathcal{L}(X)\) which is densely Li-Yorke chaotic but does not admit a dense irregular manifold.
\end{theorem}

\begin{proof}
Let \(T: X \to X\) be an Augé-Tapia operator of type 2, where \(V = \mathrm{span}(e_{1}, e_{2})\). Consider the sequence of continuous linear operators \((T_{n})_{n}\) on \(X\) defined by 
\[
    T_{n}: X \longrightarrow X, \quad x \longmapsto T^{n}x - x.
\]
To show that \((T_{n})_{n}\) is densely Li-Yorke chaotic, observe that
\[
    \mathrm{Rec}(T) = \{x \in X : \inf_{n \in \mathbb{N}} \|T_{n}x\| = 0\},
\]
and
\[
    A_{T} \subset \{x \in X : \sup_{n \in \mathbb{N}} \|T_{n}x\| = \infty\}.
\]
Thus, by Proposition \ref{G_de} and Theorem \ref{equiv-den-LY}, \((T_{n})_{n}\) is densely Li-Yorke chaotic.

Now, suppose that \((T_{n})_{n}\) admits a dense irregular manifold \(E \subset X\), leading to a contradiction. 

Since \(E \subset \mathrm{Rec}(T)\) and is dense, we obtain \(\mathbb{P}(E) = V\). In other words, there exist two irregular vectors \(p, q \in E\) with \(\mathrm{span}(\mathbb{P}(p), \mathbb{P}(q)) = V\). Consider a vector \(x \in V\setminus F \subset V\); then there exist two nonzero complex numbers \(\alpha\) and \(\beta\) such that \(x \in \mathbb{P}(\alpha p + \beta q)\). Thus, \(\alpha p + \beta q \in A_{T}\), which is a contradiction.
\end{proof}

 In \cite[Section 3]{grivaux2025questions}, Sophie Grivaux et al. established an elegant result for recurrent operators. Specifically, let \(X\) be an infinite-dimensional separable Banach space. For each \(N \in \mathbb{N}\), there exists \(T \in \mathcal{L}(X)\) with two key properties. The first property ensures that \(\mathrm{Rec}\left(\bigoplus_{i=1}^{N} T\right) = X^{N}\), implying that \(\mathrm{Rec}\left(\bigoplus_{i=1}^{\ell} T\right)\) is dense lineable for every \(\ell \in \{1, \ldots, N\}\). The second property states that \(\bigoplus_{i=1}^{N+1} T\) is not recurrent, meaning \(\mathrm{Rec}\left(\bigoplus_{i=1}^{N+1} T\right)\) is not dense in \(X^{N+1}\). In other words, such a set cannot be dense-lineable. The authors also show that this result holds for \(\mathcal{AP}\mathrm{Rec}(T)\); see \cite[Corollary 4.1]{grivaux2025questions}.

\begin{theorem}\label{mainthm}
	Let \( X \) be an infinite-dimensional separable \(F\)-space, \( (Y_{i})_{i\in \mathbb{N}} \) TVSs, and a family of sequences of linear operators \(\left((T_{i,n}: X \to Y_{i})_{n \in \mathbb{N}}\right)_{i \in \mathbb{N}} \). Suppose \( (\Psi_{i})_{i\in \mathbb{N}} \) is a family of D-phenomena from \( X \) to \( Y_{i} \), and let \( (\mathcal{F}_{i})_{i\in \mathbb{N}} \) be a family of proper Furstenberg family. If, for every \( m, i \in \mathbb{N} \), the set
	\begin{align}\label{resi}
		\mathcal{F}_{i}\Psi_{i}((T_{i,n})_{n}: m)
	\end{align}
	contains a residual subset of \( X^m \), then 
	\begin{align*}
\bigcap_{i\in \mathbb{N}}\mathcal{F}_{i}\Psi_{i}((T_{i,n})_{n}: m)
	\end{align*}
	is dense-lineable for each \( m\in \mathbb{N} \).
\end{theorem}

Before proceeding with the proof of our main statement, the first point to note is the condition in (\ref{resi}) regarding the sets containing a residual set. This raises an evident question: is this condition reasonable? The following two remarks aim to address this question, suggesting that it is indeed well-justified.

\begin{remark}
In \cite{bernal1999densely}, L. Bernal considers a sequence of operators \(T_{n}: X \to Y\), where \(X\) and \(Y\) are separable and metrizable spaces. Using an elegant technique, the author demonstrates that if the sequence \((T_{n})_{n}\) is densely hereditary hypercyclic (DHHC), then the set \(\mathrm{HC}((T_{n})_{n})\) is dense-lineable. This result is achieved by inductively leveraging the density of \(\mathrm{HC}((T_{n_{j}})_{j})\) for each subsequence \((n_{j})\). 

Now, consider a family of sequences of operators \(\left((T_{i,n}: X \to Y)_{n \in \mathbb{N}}\right)_{i \in \mathbb{N}}\), where for each \(i \in \mathbb{N}\), the sequence \((T_{i,n})_{n}\) is DHHC. A natural question arises: is the set 
\begin{align}\label{dense-HC}
    \bigcap_{i\in \mathbb{N}}\mathrm{HC}\left((T_{i,n})_{n}\right)
\end{align}
dense-lineable? 

When attempting to replicate Bernal's ideas in this context, a difficulty arises in the \(j\)-th inductive step. Specifically, we cannot guarantee that the set 
\[
\bigcap_{i \in \mathbb{N}}\mathrm{HC}((T_{i, p_{i}(j,n)})_{n})
\]
is dense. However, this issue can be resolved if \(X\) is additionally assumed to be complete. In such a case, Theorem \ref{DHHC} ensures that for each \(m \in \mathbb{N}\), the \(m\)-fold product of \((T_{n})_{n}\) is DHHC. Moreover, in \cite{godefroy1991operators, Bes-Peris--Hereditarily-Hypercyclic-Operators--JFA-1999}, it is shown that the density of hypercyclic vectors implies that they form a residual set. In other words, \(\mathrm{HC}((T_{i,n})_{n}:m)\) is residual for each \(i, m\); that is, this corresponds to our condition (\ref{resi}). Following the approach of L. Bernal, the Baire category theorem guarantees that the set \( \bigcap_{i \in \mathbb{N}} \mathrm{HC}((T_{i, p_{i}(j,n)})_{n}) \) is dense. Hence, we can affirmatively answer (\ref{dense-HC}) for each \(m\)-fold product using L. Bernal's technique. In conclusion, Theorem \ref{mainthm} ensures a common dense-lineability without strongly relying on the universality property of the sequence \((T_{i,n})_{n}\) for each \(i\).
\end{remark}

\begin{remark}
Recently, in \cite{grivaux2025questions, manuel}, the following result was independently established. Let \(X\) be a separable \(F\)-space and let \(T: X \to X\) be a continuous linear operator. Then, \(T\) is quasi-rigid if and only if, for each \(n \in \mathbb{N}\), \(\bigoplus_{i=1}^{n} T\) is recurrent. This can be reformulated as follows: \(T\) is quasi-rigid  if and only if for each \(n\in \mathbb{N}\), the set \(\mathrm{Rec}(\bigoplus_{i=1}^{n} T)\) is residual. Hence, condition (\ref{resi}) becomes evident. 
\end{remark}

\begin{proof}[Proof of Theorem \ref{mainthm}]
Since for each \( m\in \mathbb{N} \) the following set
\begin{align*}
    \bigcap_{i\in \mathbb{N}} \mathcal{F}_{i}\Psi_{i}\left((T_{i,n})_{n}: m\right)
\end{align*}
contains a residual subset of \(X^{m}\). According to Mycielski Theorem, there exists a Mycielski set \(\mathcal{K} \subset X\) such that for each \( m \in \mathbb{N} \), we have the following inclusion:
\begin{align*}
\mathcal{K}^{m} \setminus \Delta_{m} \subset  \bigcap_{i\in \mathbb{N}} \mathcal{F}_{i}\Psi_{i}\left((T_{i,n})_{n}: m\right),
\end{align*} 	
where \(\Delta_{m}\) denotes the diagonal in \( X^{m} \).

For any fixed \(\ell \in \mathbb{N}\), we observe that for any non-empty open set \(U \subset X\), it is possible to choose linearly independent vectors \(\{q_{1}, \ldots, q_{\ell}\} \subset U \cap \mathcal{K}\).

Now, consider \(\{V_{m}\}_{m \in \mathbb{N}}\), a countable basis of open sets in \(X\). For \(V_{1}\), we can select \(\{x_{1,1}, \ldots, x_{1,\ell}\} \subset V_{1} \cap \mathcal{K}\) as linearly independent vectors. Next, for \(V_{2}\), note that the set \(V_{2} \setminus \mathrm{span}(\{x_{1,t}\}_{t=1}^{\ell})\) is a non-empty open set since \(X\) is infinite-dimensional. Consequently, we can choose \(\{x_{2,1}, \ldots, x_{2,\ell}\} \subset V_{2} \setminus \mathrm{span}(\{x_{1,t}\}_{t=1}^{\ell}) \cap \mathcal{K}\) as linearly independent vectors. Proceeding inductively, we obtain 
\[
\{x_{m,j}: m \in \mathbb{N}, \, j \in \{1, \ldots, \ell\}\},
\]
which are linearly independent vectors. Furthermore, for each \(1 \leq j \leq \ell\), the set \(E_{j} := \{x_{m,j}\}_{m \in \mathbb{N}}\) is dense in \(X\).

In other words, we can construct pairwise disjoint countable dense subsets \(\{E_{j}\}_{j=1}^{\ell}\) contained in \(\mathcal{K} \setminus \{0\}\), such that \(\bigcup_{j=1}^{\ell} E_{j}\) forms a linearly independent set.

We claim that 
\begin{align*}
\mathrm{span}\left(\prod_{j=1}^{\ell} E_{j}\right) \setminus \{0\} \subset \mathcal{F}_{i}\Psi_{i}\left((T_{i,n})_{n}: \ell\right), \forall i\in \mathbb{N}.
\end{align*}
To show this, fix \(i\in \mathbb{N}\). Let \( m \in \mathbb{N} \), \( \{a_{k}\}_{k=1}^{m} \subset \mathbb{C}\setminus \{0\} \), and \( \{y_{k}\}_{k=1}^{m} \subset \prod_{j=1}^{\ell} E_{j} \), where each \( y_k \) is expressed as 
\begin{align*}
y_{k} = (z_{k,1}, \ldots, z_{k,\ell}),
\end{align*}
with \( z_{k,j} \in E_{j} \) for all \( k \).
		
Let \( q = \sum_{k=1}^{m} a_{k} y_{k} \). Clearly, \( q \neq 0 \). Moreover, \( q \) can be written as 
\begin{align*}
q = \left(q_{1}, \ldots, q_{\ell}\right)\;\; \text{with}\quad q_{i} = \sum_{k=1}^{m} a_{k} z_{k,j}.
\end{align*}		
For any \(\displaystyle{\prod_{j=1}^{\ell} A_{j,i} \in \Psi_{i}(q)}\), where \( A_{j,i} \in \Psi_{i}\left(\sum_{k=1}^{m} a_{k} z_{k,j}\right)\) for each \( j \), the \(D\)-phenomenon \(\Psi_{i}\) implies the existence of non-empty open sets \( W_{k,j,i} \subset X \) such that \( W_{k,j,i} \in \Psi_{i}(z_{k,j}) \) and 
\begin{align*}
\sum_{k=1}^{m} a_{k} W_{k,j,i} \subset A_{j,i}.
\end{align*}
Next, since \( (z_{k,j})_{k,j} \in \mathcal{F}_{i}\Psi_{i}\left((T_{i,n})_{n}: m\times \ell\right) \), we have
		\begin{align*}
			\{ n \in \mathbb{N}: T_n z_{k,j} \in W_{k,j,i}, \ \forall k,j \} \in \mathcal{F}_{i}.
		\end{align*}
		Thus, we obtain the following inclusions:
		\begin{align*}
			\left\{ n \in \mathbb{N}: T_{i,n} z_{k,j} \in W_{k,j,i}, \ \forall k,j \right\} 
			& \subseteq \left\{ n \in \mathbb{N}: \sum_{k=1}^{m} a_{k} T_{i,n} (z_{i}) \in \sum_{k=1}^{m} a_{k} W_{k,j,i}, \ \forall j \right\} \\
			& \subseteq \left\{ n \in \mathbb{N}: (T_{i,n}(q_1), \ldots, T_{i,n}(q_{\ell})) \in \prod_{j=1}^{\ell} A_{j,i} \right\}.
		\end{align*}
		This implies that the last set in the previous inclusions belongs to \(\mathcal{F}_{i}\), i.e., \( q \in \mathcal{F}_{i}\Psi_{i}\left((T_{i,n})_{n}: \ell\right)\) for all \(i\in \mathbb{N}\). This concludes the proof.
\end{proof}

\begin{corollary}
Let \(X\) be an infinite-dimensional separable \(F\)-space, and let \(\Lambda \subset \mathcal{L}(X)\) be a countable collection of quasi-rigid operators. Then,
\[
\bigcap_{T \in \Lambda} \mathrm{Rec}\left(T\right)
\]
is dense lineable for each \(m\in \mathbb{N}\).
\end{corollary}

\begin{proof}
According to \cite{grivaux2025questions, manuel}, for each \(m \in \mathbb{N}\) and \(T \in \Lambda\), the set
\begin{align*}
    \mathrm{Rec}(\bigoplus_{i=1}^{m}T) 
\end{align*}
is residual in \(X^{m}\). The proof follows from Theorem \ref{mainthm}.
\end{proof}

\begin{proposition}
Let \(X\) be an infinite-dimensional separable \(F\)-space, and let \(\Lambda \subset \mathcal{L}(X)\) be a countable collection. Suppose that for each \(T \in \Lambda\), there exists an operator \(h_{T} \in \mathcal{L}(X)\) and a sequence \((\theta_{n,T})_{n} \uparrow \infty\) such that 
\begin{align}\label{seque-dense}
    \{x \in X : \lim_{n \to \infty} T^{\theta_{n,T}} x = h_{T}x\}
\end{align}
is dense in \(X\) for each \(T \in \Lambda\). Then there exists a dense vector subspace \(E\) of \(X\), and for each \(T \in \Lambda\), there exists a sequence \((k_{n,T})_{n} \uparrow \infty\) such that 
\begin{align*}
    E \subset \{x \in X : \lim_{n \to \infty} T^{k_{n,T}} x = h_{T}x\}, \quad \forall\, T \in \Lambda.
\end{align*}
\end{proposition}

\begin{proof}
 For any fixed \(T \in \Lambda\), the density of the set in (\ref{seque-dense}) implies that the set \(\Psi^{h_{T}}\left(\bigoplus_{j=1}^{m} T\right)\) is residual in \(X^{m}\) for each \(m \in \mathbb{N}\).

Now, for each \(m \in \mathbb{N}\), consider the residual set \(\mathcal{R}_{m} \subset X^{m}\) defined by
\begin{align*}
 \mathcal{R}_{m} := \bigcap_{T \in \Lambda} \Psi^{h_{T}}\left(\bigoplus_{j=1}^{m} T\right).
\end{align*}
By Mycielski's theorem, there exists a Mycielski set \(\mathcal{K} \subset X\) such that \(\mathcal{K}^{m} \subset \mathcal{R}_{m}\) for all \(m \in \mathbb{N}\). Due to the separability of \(X\), there exists a countable dense set \(\{x_{i}\}_{i \in \mathbb{N}} \subset \mathcal{K}\). As in the proof of \cite[Theorem 3.3]{manuel}, we can choose a sequence \((k_{n,T})_{n} \uparrow \infty\) such that 
\[
\lim_{n \to \infty} T^{k_{n,T}} x_{i} = h_{T} x_{i}
\]
for all \(i \in \mathbb{N}\) and for all \(T \in \Lambda\). To conclude, it suffices to consider \(E := \mathrm{span}\{x_{i} : i \in \mathbb{N}\}\).
\end{proof}

\subsection*{Acknowledgements}

The first author was partially supported by CAPES, CNPq-Projeto Universal, FAPERJ "Cientista do Nosso Estado" E-26/201.181/2022 and Pronex-Dynamical Systems. The second author was partially supported by CAPES.

\bibliographystyle{amsplain}
\bibliography{dense-lineable}

\end{document}